\documentclass[10pt]{amsart}
\usepackage{amsmath,amssymb,latexsym,soul,cite,mathrsfs}

\usepackage{color,enumitem,graphicx}
\usepackage[colorlinks=true,urlcolor=blue,
citecolor=red,linkcolor=blue,linktocpage,pdfpagelabels,
bookmarksnumbered,bookmarksopen]{hyperref}
\usepackage[english]{babel}

\usepackage[left=3cm,right=3cm,top=3cm,bottom=3cm]{geometry}

\usepackage[hyperpageref]{backref}

\numberwithin{equation}{section}

\newtheorem{theorem}{Theorem}[section]
\newtheorem{lemma}[theorem]{Lemma}

\newtheorem{proposition}[theorem]{Proposition}

\newtheorem{remark}[theorem]{Remark}

\newcommand{\R}{{\mathbb R}}
\renewcommand{\S}{{\mathbb S}}

\newcommand{\eps}{\varepsilon}
\renewcommand{\epsilon}{\varepsilon}

\renewcommand{\O}{{\mathcal  O}}
\renewcommand{\theta}{{\vartheta}}
\renewcommand{\H}{{\mathcal H}}
\renewcommand{\rightarrow}{\to}

\newcommand{\ud}{\mathrm{d}}
\newcommand{\N}{\mathbb{N}}
\newcommand{\dive}{\mathrm{div}}

\title[Fractional problems with vanishing potentials]{Critical and subcritical fractional \\ problems with vanishing potentials}

\author[J.M.\ do \'O]{Jo\~ao Marcos do \'O}
\author[O.H.\ Miyagaki]{Ol\'{i}mpio H.\ Miyagaki}
\author[M. Squassina]{Marco Squassina}

\address[J.M. do \'O]{\newline\indent Department of Mathematics
\newline\indent 
Federal University of Para\'{\i}ba
\newline\indent
58051-900, Jo\~ao Pessoa-PB, Brazil}
\email{\href{mailto:jmbo@pq.cnpq.br}{jmbo@pq.cnpq.br}}

\address[O.\ Miyagaki]{\newline\indent Department of Mathematics 
\newline\indent 
 Federal University of Juiz de Fora
\newline\indent 
36036-330  Juiz de Fora, Minas Gerais, Brazil}
\email{\href{mailto:ohmiyagaki@gmail.com}{ohmiyagaki@gmail.com}}

\address[M.\ Squassina]{\newline\indent Dipartimento di Informatica \newline\indent
Universit\`a degli Studi di Verona
\newline\indent
C\'a Vignal 2, Strada Le Grazie 15, I-37134 Verona, Italy}
\email{\href{mailto:marco.squassina@univr.it}{marco.squassina@univr.it}}

\thanks{Research supported in part by INCTmat/MCT/Brazil, J.M.\ do \'O was partially supported by CNPq, M.\ Squassina was supported by
MIUR project Variational and Topological Methods in the Study of Nonlinear Phenomena, O.H.\ Miyagaki
 was partially supported by CNPq/Brazil and  CAPES/Brazil (Proc 2531/14-3)}
 
 \thanks{The paper was completed while the second author was visiting the
  Department of Mathematics of Rutgers University, whose hospitality he gratefully
 acknowledges. He would also like to express his gratitude to Prof.\ Haim Brezis.}
 

\begin{document}

\begin{abstract}
We investigate a class of nonlinear nonautonomous scalar field equations with fractional diffusion,
critical power nonlinearity and a subcritical term. The involved potentials are allowed for vanishing
behavior at infinity. The problem is set on the whole space and compactness issues have to be tackled.
\end{abstract}
\maketitle




\section{Introduction and main results}
\noindent
We consider existence of solutions for the following class of equations 
\begin{equation}
\label{PS}
(-\Delta)^{\frac{s}{2}} u + V(x)u = K(x)f(u) + \lambda |u|^{2^{*}_{s}-2}u \quad\,\, \text{in $\R^N$}.
\end{equation}
Here $ \lambda\geq 0$ is a parameter,  $s \in (0,2),$ $2^{*}_{s}=2N/(N-s)$, $N>s$, $(-\Delta)^{\frac{s}{2}} $ is fractional laplacian 
$V,K$ are positive functions and $f$ is a continuous function with quasicritical growth.
Recently, a great attention has been focused on the study nonlinear problems involving fractional  laplacian,
 in view of real-world applications. For instance, this type of operators arise in thin obstacle problems, optimization, finance, phase
transitions, stratified materials, anomalous diffusion, crystal dislocation, soft thin films, semipermeable membranes, flame
propagation, conservation laws, ultra-relativistic limits of quantum mechanics, quasi-geostrophic flows, multiple scattering,
minimal surfaces, materials science and water waves, see \cite{nezza}. 
The fractional  laplacian $(-\Delta)^{\frac{s}{2}}$ with $s \in (0,2)$ of a function $\phi:\R^N \to \R$ is  defined  by  
$$
{\mathcal F}((-\Delta)^{\frac{s}{2}}\phi)(\xi)=|\xi|^{s}{\mathcal F}(\phi)(\xi),
 \quad\text{for $s \in(0,2)$},
 $$ 
 where ${\mathcal F}$ is the Fourier transform.
We are going explore problem \eqref{PS} with {\it zero mass} potential, that is when 
$V(x)\to 0$, as $|x|\to\infty$. This class was studied by several researchers in the local case $s=2$, e.g.\
in \cite{AS2012,AW,GM,BGM,BL,BPR,LW,Montenegro} and reference therein, where the main feature is to impose restrictions on $V,K$ to get  
some compact embedding into a weighted $L^p$ space. Recently 
Alves and Souto, in \cite{AS}, in addition to  improving all the former restrictions on the potentials, handled subcritical nonlinearities
$f$ which do not satisfy the so-called Ambrosetti-Rabinowitz condition, namely, 
\begin{equation}
\text{there exists $\theta\in (2,2^*_s)$ with
           $0<\theta F(s)\leq sf(s)$ for all $s>0$, \,\,\, $F(s)=\int_{0}^{s}f(t)dt$.}
\tag{AR}
\end{equation}
Conditions weaker than (AR) were used, first time,  in \cite{costa-1,jean,zou,LW,LWa}.
In all the above cited papers, the nonlinearity $f$ had subcritical growth, that is, in addition to $ \lambda =0,$ the growth of $f$ in 
comparable with $s^p$ with $p<(N+2)/(N-2)$, for $N \geq 3$. 
In the case $s\in (0,2)$, nonlocal case, we say that  $f$ has a subcritical growth, if the growth of $f$ in $s$ 
is comparable with $s^p$ for $p<(N+s)/(N-s)$, with $N>s$. 
In this situation, we would like to mention two works, one by Chang and  Wang \cite{CW}, 
where the authors recovered the  Berestycki and Lions\cite{BL} results by improving Strauss compactness result \cite{St}, and
a paper by Secchi \cite{Simone}
where the existence of ground state solutions is established. 
Motivated by the papers above, we are going to study the nonlocal case, with nonlinearities involving a critical growth and
a subcritical perturbation $f$. 
Elliptic problems with critical growth, after the pioneering works by Brezis and Nirenberg\cite{BN} have had many progresses in several 
directions. We would like
to mention \cite{ambrosetti,willem} and the references therein, in local case. 
For nonlocal case, in bounded domain, we cite \cite{tanhalf, FW,servadei2,cabretan,Jin,barrios} and references therein, while in
 whole space was studied recently in \cite{SZY} for non vanishing potential. 
Recently, Caffarelli and Silvestre \cite{caffarelli} developed a local interpretation of the fractional
Laplacian given in $\R^N$ by considering a Neumann type operator in the extended domain
$\R^{N+1}_{+}$ defined by  $\{(x, t) \in \R^{N+1} : t > 0\}$. A similar extension, for nonlocal problems on bounded domain with
the zero Dirichlet boundary condition,  was established, for instance,  by Cabr\`e and Tan in 
\cite{cabretan}, Tan \cite{tan},  Capella, D\`avila, Dupaigne and Sire \cite{capela},  Br$\ddot{\mbox{a}}$ndle, 
Colorado, de Pablo and S\`anchez \cite{colorado}. It is worth noticing that, in a bounded domain, the Fourier definition
of the fractional laplacian and its local Caffarelli-Silvestre
interpretation do not agree, see the discussion developed  \cite{servadei} for more details.
For $u \in H^{s/2}(\R^N),$ the solution $w \in X^{s}(\R^{N+1}_{+})$ of  
 \begin{equation}
 \left\{ \begin{array}{rcl}
 -\dive ( y^{1-s}\nabla w)=0 & \mbox{in}&  \R^{N+1}_{+}\noindent\\
  w=u& \mbox{on}& \R^{N} \times\{0\}\noindent
 \end{array}\right. 
 \end{equation}
is called $s$-harmonic  extension $w=E_{s}(u)$ of $u$ and it is proved in \cite{caffarelli} (see also \cite{colorado}) that
$$
\lim_{y \to 0^+} y ^{1-s}\frac{\partial w}{\partial y}(x,y)=-\frac{1}{k_{s}}(-\Delta)^{\frac{s}{2}}u(x),
$$
where  
$$
k_{s}=2^{1-s}\Gamma(1-\frac{s}{2})/\Gamma (\frac{s}{2}).
$$ 
Here the spaces $X^{s}(\R^{N+1}_{+})$ and $H^{s/2}(\R^N)$
are defined as the completion of $C^{\infty}_{0}(\overline{\R^{N+1}_{+}})$ and  $C^{\infty}_{0}(\R^{N}),$
under the norms (which actually do coincide, see \cite[Lemma A.2]{colorado})
\begin{align*}
\|w\|_{X^{s}}:=&\Big(\int_{ \R^{N+1}_{+}}k_{s}y^{1-s}|\nabla w|^2 \ud x\ud y\Big)^{1/2},   \\
\|u\|_{H^{\frac{s}{2}}}:=&\Big(\int_{\R^N}|2\pi\xi|^{s}|\mathbb{F}(u(\xi))|^2 \ud \xi\Big)^{1/2}=\Big(\int_{\R^N}|(-\Delta)^{\frac{s}{2}}u|^2 \ud x\Big)^{1/2}.
\end{align*}
\noindent
Our problem \eqref{PS} will be studied in the  half-space, namely,
 \begin{equation}\label{NPS}
 \left\{ \begin{array}{rcl}
 -\dive ( y^{1-s}\nabla w)=0 & \mbox{in}&  \R^{N+1}_{+}\noindent\\
  -k_{s}\frac{\partial w}{\partial \nu}=-V(x)u +K(x)f(u) +  \lambda  |u|^{2^{*}_{s}-2}u & \mbox{on}& \R^{N} \times\{0\},\noindent
 \end{array}\right. 
 \end{equation}
where 
$$
\frac{\partial w}{\partial \nu}=\lim_{y \to 0^+} y ^{1-s}\frac{\partial w}{\partial y}(x,y).
$$ 
We are looking for a positive solution in the Hilbert space $E$ defined by
$$
E=\Big\{ w \in X^{s}(\R^{N+1}_{+}): \ \int_{\R^N}V(x)w(x,0)^2 \ud x < \infty\Big\}
$$
endowed with norm
$$
\|w\|:= \Big(\int_{ \R^{N+1}_{+}}k_{s}y^{1-s}|\nabla w|^2 \ud x\ud y + \int_{\R^N}V(x)w(x,0)^2 \ud x\Big)^{1/2}.
$$
Consider the Euler-Lagrange functional associated to \eqref{NPS}  given by
\begin{equation}\label{functional}
J_ \lambda (w):=\frac{1}{2}\|w\|^2 - \int_{\mathbb{R}^N} K(x) F(w(x,0))   \ud x  
- \frac{\lambda}{2^{*}_{s}}\int_{\mathbb{R}^N} w^+(x,0)^{2^{*}_{s}}   \ud x 
\end{equation}
which is $C^1$ with G\^ateaux derivative
\begin{align}
\label{derivative}
 J_ \lambda '(w)v&=\int_{ \R^{N+1}_{+}}k_{s}y^{1-s}\nabla w \cdot\nabla v \ud x\ud y + \int_{\R^N}V(x)w(x,0)v(x,0) \ud x \\
& - \int_{\mathbb{R}^N} K(x) f(w(x,0))v(x,0)  \ud x -  \lambda  \int_{\mathbb{R}^N} w^+(x,0)^{2^{*}_{s}-1} v(x,0) \ud x, 
\,\,\quad\text{for all $w,v\in E.$}   \notag
\end{align}
We now formulate assumptions for $V,K,f$ in problem~\eqref{PS}.
\vskip4pt
\noindent
$\bullet$ {\sc Assumptions on $V$ and $K$.} 
\begin{description}
\item[(I) (sign of $V$ and $K$)]  $V,K$ are continuous, $V,K>0$ on $\R^N$ and $K \in L^{\infty}(\R^N)$;
\item[(II) (decay of $K$)] If $\{A_n\}$ is a sequence of Borel sets of $\R^N$ with $|A_n|\leq R$ for some $R>0$, 
\begin{equation}
\label{decay-K}
\lim_{r\to \infty} \int_{A_n \cap B^{c}_{r}(0)} K(x)\ud x =0,\quad\text{uniformly with respect to $n \in \N$;}
\end{equation}
\item[(III) (interrelation between $V$ and $K$)] either 
\begin{equation}
\label{primaint}
K/V\in  L^{\infty}(\R^N)
\end{equation}
or  there exists $p \in (2,2^{*}_{s})$ such that 
\begin{equation}
\label{secondainter}
\lim_{|x|\to\infty}\frac{K(x)}{V(x)^\gamma}=0,
\qquad \gamma=\frac{ps-N(p-2)}{2s}\in (0,1).
\end{equation}
\end{description}
\vskip4pt
\noindent
$\bullet$ {\sc Assumptions on $f$.} 
\begin{description}
\item[(f1) (behaviour at zero)]  $f:\R\to\R^+$ is continuous with $f=0$ on $\R^-$.
If \eqref{primaint} holds, then 
$$
\limsup_{s\to 0^+}\frac{f(s)}{s}=0.
$$
If condition \eqref{secondainter} holds, we assume
$$
\limsup_{s\to 0^+}\frac{f(s)}{s^{p-1}}< +\infty.
$$
\item[(f2) (quasi-critical growth)] If \eqref{primaint} holds, then 
$$
\limsup_{s\to +\infty} \frac{f(s)}{s^{2^{*}_{s}-1}}=0.
$$
If condition \eqref{secondainter} holds, we assume
$$
\limsup_{s\to +\infty}\frac{F(s)}{s^{p}}< +\infty.
$$
\item[(f3) (super-quadraticity)]  $\frac{f(s)}{s}$ is non-decreasing in $\R^+$, and 
$$
\limsup_{s\to +\infty} \frac{F(s)}{s^{2}}=+\infty.
$$
\item[(f3)$'$ (super-quadraticity)]  $\frac{f(s)}{s}$ is non-decreasing in $\R^+$ and there exist $C_0>0$ and $q \in (2,2^{*}_{s})$ with
$$
F(s)\geq C_0 s^q, \,\,\quad \text{for all $s\in\R^+$}
$$
\end{description}
\noindent

\vskip3pt
\noindent
The following are the main results of the paper.


\begin{theorem}Assume {\rm (I)}-{\rm (III)}, {\rm (f1)-(f3)} and $\lambda=0$. Then \eqref{PS} admits a positive solution $u\in E$.
\label{mainsub}
\end{theorem}

\begin{theorem}
\label{main}
Assume {\rm (I)}-{\rm (III)}, {\rm (f1)-(f2)-(f3)$'$}, (AR), $\lambda=1$ and that one of the following hold
\begin{enumerate}
\item $N>2s$, 
\item $N=2s$, 
\item $s<N<2s$ and $q>\frac{N}{N-s}$,
\item $s<N<2s$ and $q<\frac{N}{N-s},$ with $C_0$ large enough.
\end{enumerate}
Then \eqref{PS} admits a positive solution $u\in E$.
\end{theorem}

\noindent
Throughout the paper, unless explicitly stated, the symbol $C$ will always denote a generic positive constant,
which may vary from line to line.

\smallskip
\section{Preliminary results}

\noindent
Consider the weighted Banach space:
$$
L^{p}_{K}=\Big\{ u:\R^N \to \R\,\, \mbox{measurable and}\ \int_{\R^N}K(x)|u|^{p} \ud x < \infty\Big\},\quad
\|\cdot\|_{L^{p}_{K}}=\Big(\int_{\R^N}K(x)|u|^{p} \ud x\Big)^{1/p}.
$$
\noindent
The first result, on compact injections for $E$, follows by adapting the arguments in \cite{AS}. 

\begin{proposition}[Compactness]\label{converge}
The following facts hold:
\begin{enumerate}
\item $E$ is compactly embedded  into $L^q_{K}$  for all $q\in (2,2^{*}_{s}),$ provided that \eqref{primaint} holds;
\item $E$ is compactly embedded  into $L^p_{K}$ provided that \eqref{secondainter}  holds;
\item If $w_n \rightharpoonup w$ in $E$, then up to a subsequence,
\[
\lim_n\int_{\mathbb{R}^N} K(x) F(w_n(x,0))  \ud x  =\int_{\mathbb{R}^N} K(x) F(w(x,0))  \ud x ;
\]
\item If $w_n \rightharpoonup w$ in $E$, then up to a subsequence,
\[
\lim_n\int_{\mathbb{R}^N} K(x) w_n(x,0) f(w_n(x,0))   \ud x  =\int_{\mathbb{R}^N} K(x) w(x,0) f(w(x,0))  \ud x ;
\] 
\item If $w_n \rightharpoonup w$ in $E$, then, up to a subsequence, for any $v\in E$,
\[
\lim_n\int_{\mathbb{R}^N} w_n^+(x,0)^{2^{*}_{s}-1}v(x,0)  \ud x  =\int_{\mathbb{R}^N} w^+(x,0)^{2^{*}_{s}-1}v(x,0)   \ud x.
\] 
\item If $w_n \rightharpoonup w$ in $E$, then up to a subsequence, for any $v\in E$,
\[
\lim_n\int_{\mathbb{R}^N} K(x)f(w_n(x,0))v(x,0)   \ud x  =\int_{\mathbb{R}^N} K(x) f(w(x,0))v(x,0)  \ud x .
\] 
\end{enumerate}
\end{proposition}

\begin{proof}
Assume that condition \eqref{primaint} holds, let $q\in (2,2^{*}_{s})$ and let us prove assertion (1). 
Let $\epsilon >0$. Then, there exist $0<s_0(\eps)<s_1(\eps)$, a positive constant $C(\eps)$ and $C_0$ depending only on $V$ and $K$, such that
\begin{equation}
\label{2.4}
 K(x)|s|^q\leq \epsilon C_0 (V(x)|s|^2 + |s|^{2^{*}_{s}})+C(\eps)K(x)\chi_{[s_0(\eps),s_1(\eps)]}(|s|)|s|^{2^{*}_{s}},\quad\text{for all $s\in \R$}.
\end{equation}
Therefore we obtain, for every $w \in E$ and $r>0$,
\begin{equation}\label{2.5}
 \int_{B^{+^{c}}_{r}(0)\cap \{y=0\}} K(x)|w(x,0)|^q \ud x \leq \epsilon Q(w)+ C(\eps)s_1(\eps)^{2^*_s} \int_{A_\eps\cap (B^{+^{c}}_{r}(0)\cap \{y=0\})} K(x) \ud x,
\end{equation}
where we have set
\begin{equation}
\label{Q-A}
Q(w):=C_0\int_{\R^N} (V(x)|w(x,0)|^2 + |w(x,0)|^{2^{*}_{s}}) \ud x,
\quad
A_\eps:=\big\{x\in \R^N:  s_0(\eps) \leq |w(x,0)|\leq s_1(\eps)\big\}.
\end{equation}
If $(w_n)\subset E$ is such that $w_n \rightharpoonup w$ weakly in $E$ for some $w\in E$, there exists $M>0$ with
\begin{equation}
\begin{split}
\label{boundednesses}
 \int_{\R^{N+1}_{+}}k_s |\nabla w_n|^2 \ud x \ud y + \int_{\R^N}V(x)|w_n(x,0)|^2 \ud x \leq M, \quad \text{for all $n\in \N$}, \\ 
  \int_{\R^N}|w_n(x,0)|^{2^{*}_{s}} \ud x\leq M,\quad \text{for all $n\in \N$},
\end{split}
\end{equation}
so that $Q(w_n)$ is bounded in $\R$. On the other hand, 
if $A^n_\eps=\big\{s_0(\eps) \leq |w_n(x,0)|\leq s_1(\eps)\big\},$ we get 
$$
s_0(\eps)^{2^{*}_{s}}|A^n_\eps|\leq \int_{A^n_\eps}|w_n(x,0)|^{2^{*}_{s}}\ud x \leq \int_{\R^N}|w_n(x,0)|^{2^{*}_{s}}\ud x\leq M, \quad \text{for all $n\in \N$}. 
$$
which implies that $\sup_{n\in \N} |A^n_\eps|<+\infty$.
Then, in light of \eqref{decay-K}, there exists $r(\eps)>0$ such that
\begin{equation}\label{2.6}
 \int_{A^n_\eps\cap (B^{+^{c}}_{r(\eps)}(0)\cap \{y=0\})} K(x) \ud x <\frac{\epsilon}{C(\eps) s_1(\eps)^{2^{*}_{s}}},\quad \text{for all $n\in \N$}. 
\end{equation}
Whence, invoking \eqref{2.5}, we get
\begin{equation}
\label{2.7}
 \int_{B^{+^{c}}_{r(\eps)}(0)\cap \{y=0\}} K(x)|w_n(x,0)|^q \ud x \leq  (2C_0M+1)\eps. 
\end{equation}
By the fractional compact embedding \cite{colorado}, we have
\begin{equation}\label{2.8}
 \lim_{n\to \infty}\int_{B^{+}_{r(\eps)}(0)\cap \{y=0\}} K(x)|w_n(x,0)|^q \ud x=\int_{B^{+}_{r(\eps)}(0)\cap \{y=0\}} K(x)|w(x,0)|^q \ud x.
\end{equation}
Combining \eqref{2.7}-\eqref{2.8}, yields 
$$
\lim_{n}\int_{\R^N\cap \{y=0\}} K(x)|w_n(x,0)|^q \ud x=\int_{\R^N \cap \{y=0\}} K(x)|w(x,0)|^q \ud x,
$$
which concludes the proof of (1). 
\vskip2pt
\noindent
Assume now that condition \eqref{secondainter} holds and let us prove assertion (2).
By a direct calculation, for any $x\in \R^N$ and $s\geq 0$, if $\gamma\in (0,1)$ is
the constant introduced in \eqref{secondainter}, we get 
$$
V(x) s^{2-p} + s^{2^{*}_{s}-p}\geq \omega(p,s) V(x)^{\gamma},
\qquad \omega(p,s)=\Big(\frac{2^{*}_{s}-2}{2^{*}_{s}-p}\Big)\Big(\frac{p-2}{2^{*}_{s}-p}\Big)^{\frac{2-p}{2^{*}_{s}-2}}.
$$
Let $\eps>0$. Combining this inequality with \eqref{secondainter}, there exists $r(\eps)>0$ such that
\begin{equation}
\label{controllokv}
K(x)|s|^p \leq \epsilon \big(V(x)|s|^{2} +|s|^{2^{*}_{s}}\big), \quad\text{for all $s\in \R$ and $|x|\geq r(\eps).$}
\end{equation}
Then, for all  $w\in E$, we conclude
$$
\int_{B^{+^{c}}_{r(\eps)}(0)\cap \{y=0\}} K(x)|w(x,0)|^p \ud x\leq \epsilon \int_{B^{+^{c}}_{r(\eps)}(0)\cap \{y=0\}}
 (V(x)|w(x,0)|^{2} +|w(x,0)|^{2^{*}_{s}})\ud x.
 $$
If $(w_n)\subset E$ and $w_n \rightharpoonup w$ weakly in $E,$ there exists $M>0$ such that \eqref{boundednesses} hold. Whence,
for a suitable radius $r(\eps)>0$ there holds
\begin{equation}\label{2.10}
\int_{B^{+^{c}}_{r(\eps)}(0)\cap \{y=0\}} K(x)|w_n(x,0)|^p \ud x\leq 2\epsilon M,\quad\text{for all $n \in \N$}. 
\end{equation}
Since $p\in (2,2^{*}_{s})$, by the fractional compact embedding we have 
\begin{equation}\label{2.11}
 \lim_{n}\int_{B^{+}_{r(\eps)}(0)\cap \{y=0\}} K(x)|w_n(x,0)|^p \ud x=\int_{B^{+}_{r(\eps)}(0)\cap \{y=0\}} K(x)|w(x,0)|^p \ud x.
\end{equation}
Combining $(\ref{2.10})$ and $(\ref{2.11})$ we get
$$
\lim_{n}\int_{\R^N\cap \{y=0\}} K(x)|w_n(x,0)|^p \ud x=\int_{\R^N \cap \{y=0\}} K(x)|w(x,0)|^p \ud x,
$$
which concludes the proof of assertion (2). 
\vskip2pt
\noindent
Let us now turn to the proof of (3) and (4) under assumption \eqref{primaint}.
From $(f_1)$-$(f_3)$, fixed $q\in (2,2^{*}_{s})$ and given $\epsilon>0$, there exist $0<s_0(\eps)<s_1(\eps)$, $C(\eps)>0$ 
and $C_0$ depending only upon $V$ and $K$, with
\begin{align}\label{2.12}
 |K(x)F(s)|&\leq \epsilon C_0(V(x)|s|^2 +|s|^{2^{*}_{s}})+ C(\eps) K(x)\chi_{[s_0(\eps),s_1(\eps)]}(|s|)|s|^q, \,\,\quad\text{for all $s\in \R$}, \\
\label{2.12.1}
  |K(x)f(s)s|&\leq \epsilon C_0(V(x)|s|^2 +|s|^{2^{*}_{s}})+ C(\eps) K(x)\chi_{[s_0(\eps),s_1(\eps)]}(|s|)|s|^q, \,\,\quad\text{for all $s\in \R$}.
\end{align}
Notice that, by \eqref{decay-K}, arguing as for the proof of (1), there exists $r(\eps)>0$ such that
\begin{equation}\label{2.13}
 \int_{A^n_\eps \cap B^{+^c}_{r(\eps)}(0)\cap \{y=0\}} K(x)\ud x\leq \frac{\epsilon}{C(\eps)s_1(\eps)^q}, \quad\text{for all $n \in \N$}.
\end{equation}
Let $\{w_n\}\in E$ be bounded.
Combining the above inequality with \eqref{boundednesses} and  \eqref{2.12}-\eqref{2.12.1}, we have
\begin{align}\label{2.14}
& \int_{B^{+^{c}}_{r(\eps)}(0)\cap \{y=0\}} K(x)F(w_n(x,0)) \ud x\leq (2C_0M + 1)\epsilon ,\quad\text{for all $n \in \N$}, \\
\label{2.14.1}
& \int_{B^{+^{c}}_{r(\eps)}(0)\cap \{y=0\}} K(x)f(w_n(x,0))w_n(x,0) \ud x\leq (2C_0M + 1)\epsilon ,\quad\text{for all $n \in \N$}.
\end{align}
Since $(w_n(x,0))$ is bounded in $L^{2^{*}_{s}}(\R^N)$, by Strauss lemma \cite[Theorem A.I p.338]{BL},  we infer
\begin{align*}
 & \lim_{n}\int_{B^{+}_{r(\eps)}(0)\cap \{y=0\}} K(x)
F(w_n(x,0)) \ud x=\int_{B^{+}_{r(\eps)}(0)\cap \{y=0\}} K(x)F(w(x,0)) \ud x, \\
& \lim_{n}\int_{B^{+}_{r(\eps)}(0)\cap \{y=0\}} K(x)
f(w_n(x,0))w_n(x,0) \ud x=\int_{B^{+}_{r(\eps)}(0)\cap \{y=0\}} K(x)f(w(x,0)w(x,0) \ud x.
\end{align*}
Combining these limits with \eqref{2.14}-\eqref{2.14.1} we conclude the proof.
\vskip4pt
\noindent
Let us now turn to the proof of (3) and (4) under assumption \eqref{secondainter}. Let $\eps>0$. We learned that
there exists $r(\eps)>0$ such that  such that \eqref{controllokv} holds, yielding
\begin{align*}
& K(x)|F(s)| \leq \epsilon \big(V(x)|F(s)||s|^{2-p} +|F(s)||s|^{2^{*}_{s}-p}\big), \quad\text{for all $s\in \R$ and $|x|\geq r(\eps)$}, \\
& K(x)f(s)s \leq \epsilon \big(V(x)f(s)s|s|^{2-p} +f(s)s|s|^{2^{*}_{s}-p}\big), \quad\text{for all $s\in \R^+$ and $|x|\geq r(\eps)$}.
\end{align*}
From $(f_1)$-$(f_2)$, there exist $0<s_0(\eps)<s_1(\eps)$ satisfying
\begin{align*}
& K(x)|F(s)| \leq \epsilon \big(V(x)|s|^{2} +|s|^{2^{*}_{s}}\big), \quad\text{for all $s\in I_\eps$ and $|x|\geq r(\eps)$},  \\
& K(x)f(s)s \leq \epsilon \big(V(x)|s|^{2} +|s|^{2^{*}_{s}}\big), \quad \text{for all $s\in I_\eps\cap \R^+$ and $|x|\geq r(\eps)$},
\end{align*}
where  $I_\eps=\{\text{$|s|<s_0(\eps)$ or $|s|>s_1(\eps)$}\}.$
Then, we have 
\begin{align}\label{estimate1}
 \int_{B^{+^{c}}_{r(\eps)}(0)\cap \{y=0\}} K(x)F(w_n(x,0) \ud x \leq \epsilon Q(w_n) +C(\eps) \int_{A^n_\eps\cap (B^{+^{c}}_{r(\eps)}(0)\cap \{y=0\})} K(x) \ud x, \\
\label{estimate1-2}
 \int_{B^{+^{c}}_{r(\eps)}(0)\cap \{y=0\}} K(x)f(w_n(x,0)w_n(x,0) \ud x \leq \epsilon Q(w_n) +C(\eps) \int_{A^n_\eps\cap (B^{+^{c}}_{r(\eps)}(0)\cap \{y=0\})} K(x) \ud x,
\end{align}
where 
$$
C(\eps)=\max\Big\{\max_{[s_0(\eps),s_1(\eps)]}|F(s)|,\max_{[s_0(\eps),s_1(\eps)]}|f(s)s|\Big\}.
$$ 
Arguing as for the proof of (1), we have 
\begin{align*}
& \Big|\int_{(B^{+^{c}}_{r(\eps)}(0)\cap \{y=0\})} K(x)F(w_n(x,0)) \ud x\Big|\leq (2M+1)\epsilon, \quad\text{for all $n \in \N$},\\
& \Big|\int_{(B^{+^{c}}_{r(\eps)}(0)\cap \{y=0\})} K(x)f(w_n(x,0)) w_n(x,0) \ud x\Big|\leq (2M+1)\epsilon, \quad\text{for all $n \in \N$}.
\end{align*}
Invoking again Strauss lemma, by the above inequalities, conclusions (3) and (4) follows.
To prove (5), it is enough to observe that $w_n^+(x,0)^{2^{*}_{s}-1}\rightharpoonup w^+(x,0)^{2^{*}_{s}-1}$
weakly in $(L^{2^{*}_{s}})'$. Finally, let us prove (6). If \eqref{primaint} holds, then the sequence
$(\sqrt{K(x)} f(w_n(x,0))\chi_{\{|w_n(x,0)|\leq 1\}})$ is bounded in $L^2(\R^N)$ being
$$
|\sqrt{K(x)} f(w_n(x,0))\chi_{\{|w_n(x,0)|\leq 1\}}|^2\leq C V(x)|w_n(x,0)|^2.
$$
This, by pointwise convergence, yields for every $\varphi\in L^2(\R^N)$
$$
\lim_k\int_{\R^N} \sqrt{K(x)} f(w_n(x,0))\chi_{\{|w_n(x,0)|\leq 1\}}\varphi(x) \ud x=\int_{\R^N} \sqrt{K(x)} f(w(x,0))\chi_{\{|w(x,0)|\leq 1\}}\varphi(x) \ud x.
$$
Given $v\in E$, since $\sqrt{K(x)}\leq C \sqrt{V(x)}$, it follows $\sqrt{K(x)}v(x,0)\in L^2(\R^N)$, yielding 
\begin{equation}
\label{euno}
\lim_k\int_{\R^N} K(x) f(w_n(x,0))\chi_{\{|w_n(x,0)|\leq 1\}} v(x,0) \ud x=\int_{\R^N} K(x) f(w(x,0))\chi_{\{|w(x,0)|\leq 1\}} v(x,0) \ud x.
\end{equation}
In a similar fashion, the sequence
$(K(x) f(w_n(x,0))\chi_{\{|w_n(x,0)|\geq 1\}})$ is bounded in $L^{\frac{2^*_s}{2^*_s-1}}(\R^N)$ being
$$
|K(x) f(w_n(x,0))\chi_{\{|w_n(x,0)|\geq 1\}}|^{\frac{2^*_s}{2^*_s-1}}\leq |w_n(x,0)|^{2^*_s}.
$$
This, by pointwise convergence, and since $v\in E$, yields
\begin{equation}
\label{edue}
\lim_k\int_{\R^N} K(x) f(w_n(x,0))\chi_{\{|w_n(x,0)|\geq 1\}} v(x,0) \ud x=\int_{\R^N} K(x) f(w(x,0))\chi_{\{|w(x,0)|\geq 1\}} v(x,0) \ud x.
\end{equation}
Combining \eqref{euno}-\eqref{edue} yields the assertion. In a similar fashion one can treat the case when 
\eqref{secondainter} holds since, by means of (2), $K^{1/p}v(x,0)\in L^p(\R^N)$ for all $v\in E$ and, up to a subsequence,
\begin{equation*}
|K(x)^{\frac{p-1}{p}} f(w_n(x,0))\chi_{\{|w_n(x,0)|\leq 1\}}|^{p'}\leq K(x)|w_n(x,0)|^p\leq z(x)\in L^1(\R^N).
\end{equation*}
This concludes the proof.
\end{proof}

\noindent
From $(f_1)$-$(f_2)$ one can prove that $J_{\lambda}$ satisfies the Mountain-Pass geometry.

\begin{lemma}[Geometry]
\label{MPG} The functional $J_{\lambda}$ satisfies
\begin{enumerate}
\item There exists $\beta, \rho>0$ such that $J_{\lambda}(u)\geq \beta$ if $\|u\|=\rho$;
\item There exists $e\in E\backslash\{0\}$ with $\|u\|>\rho$ such that $J_{\lambda}(e)\leq 0$;
\end{enumerate}
\end{lemma}
\begin{proof}
 (2) is obvious. Concerning (1), observe that in light of condition \eqref{secondainter} on $V$ and $K$, 
 the space $E$ is {\em continuously} embedded into $L^p_K(\R)$ where
$p\in (2,2^*_s)$ is the precisely the value which appears in condition \eqref{secondainter}. This can be readily obtained by arguing as in the proof of
\cite[part (i) of Theorem 4]{BV} (see formula (8) therein obtained by H\"older inequality) and by using  the fractional Sobolev inequality. This is possible
since in any of the two assumptions between $V$ and $K$, we have that 
\begin{equation*}
\frac{K}{V^\gamma}\in L^\infty(\R),
\qquad \gamma=\frac{ps-N(p-2)}{2s}.
\end{equation*}
This is the fractional counterpart of the assumption on ${\mathcal W}$ in \cite{BV}.
Once this embedding is available, recall that we can write the inequality, for $\eps_0$ to be fixed small
$$
K(x)F(s)\leq \eps_0 V(x)s^2+ Cs^{2^*_s}+C K(x)s^p,\quad x\in\R,\,\, s\in\R^+,
$$
and the Mountain-Pass geometry can be proved.
\end{proof}

\noindent
Therefore, there exists a sequence $\{w_n\}\subset E$, so called {\it Cerami sequence} \cite{Ce}, such that
\begin{equation}
\label{ceramiS}
J_{\lambda}(w_n)\to c, \quad  (1+\|w_n\|)\|J_{\lambda}'(w_n)\|\to 0,
\end{equation}
where $c$  is given by
$$
c=\inf_{\gamma \in \Gamma}\max_{t\in[0,1]}J_{\lambda}(\gamma(t)),
$$
with 
$$
\Gamma=\{\gamma \in C([0,1],E): \gamma(0)=0\ \mbox{and}\ J_{\lambda}(\gamma(1))\leq 0\}.
$$


\noindent
Next we turn to the boundedness of $(w_n)$ in $E$. 
\begin{lemma}[Boundedness]
\label{cerami} 
Let $\lambda\in\{0,1\}$. Then the Cerami sequence $(w_n) \subset E$ is bounded.
\end{lemma}
\begin{proof}

First of all, we observe that $w_n^-\in E$ and, by the definition of $J_{\lambda}$, 
\begin{align*}
J_{\lambda}'(w_n)(-w_n^-)& =-\int_{ \R^{N+1}_{+}}k_{s}y^{1-s}\nabla w_n \cdot\nabla w_n^- \ud x\ud y -\int_{\R^N}V(x)w_n(x,0)w_n^-(x,0) \\
&= \int_{ \R^{N+1}_{+}}k_{s}y^{1-s}|\nabla w_n^-|^2 \ud x\ud y +\int_{\R^N}V(x)w_n^-(x,0)^2=\|w_n^-\|^2.
\end{align*}
Since $(1+\|w_n\|)J_{\lambda}'(w_n)(-\|w_n^-\|^{-1} w_n^-)=o_n(1)$ as $n\to\infty$, it follows that  $J_{\lambda}'(w_n)(-w_n^-)=o_n(1)$ as $n\to\infty$,
which in turn implies that $\|w_n^-\|=o_n(1)$, as $n\to\infty$.
\vskip3pt
\noindent{\bf Case $\lambda=0.$} Denote $J_0 = J.$ Let $t_n\in [0,1]$ be such that 
\[
J(t_nw_n)=\max_{t\in [0,1]} J(tw_n).
\]
We {\em claim} that $J(t_nw_n)$ is bounded from above.
Without loss of generality, we may assume that $t_n \in (0,1)$ for all $n$. Then, we have $J'(t_nw_n)(w_n)=0$ and 
\begin{align}
\label{computH}
2J(t_nw_n)& = 2J(t_nw_n)-J'(t_nw_n)(t_nw_n) \notag \\
&= \int_{\mathbb{R}^N} K(x)\H(t_n w_n(x,0)) \ud x  
= \int_{\mathbb{R}^N} K(x)\H(t_n w_n^+(x,0)) \ud x,  
\end{align}
where $\H(s)=sf(s)-2F(s)$ is nondecreasing and $\H=0$ on $\R^-$. 
Thus, since $t_n \in (0,1)$ and $w_n^+\geq 0$, from formula \eqref{computH} we obtain that
\begin{align*}
2J(t_nw_n)\leq & \int_{\mathbb{R}^N} K(x)\H(w_n^+(x,0)) \ud x 
= \int_{\mathbb{R}^N} K(x)\H(w_n(x,0)) \ud x   \\
=  & 2J(w_n)-J'(w_n)(w_n)=  2J(w_n)+ o_n(1),
\end{align*}
which proves the claim.
Now, we prove that $(w_n) \subset E$ is bounded.  Assume by contradiction that, up
to subsequence, $\|w_n\|\to +\infty$ as $n\to\infty$. Set $z_n:=w_n/\|w_n\|$ and suppose that
$z_n \rightharpoonup z$, as $n\to\infty$, in $E$. We now {\em claim} that $z(x,0)=0$ almost everywhere in $\R^N$. In fact,
\begin{align*}
o_n(1)+\frac{1}{2}&= \int_{\mathbb{R}^N} \frac{K(x)F(w_n(x,0))}{\|w_n\|^2}   \ud x 
=\int_{\mathbb{R}^N} \frac{K(x)F(w_n(x,0))}{|w_n(x,0)|^2} z_n^2(x,0)   \ud x.
\end{align*}
By $(f_3)$, given $\tau >0$ there exists $\xi_\tau >0$
such that $F(s)\geq \tau s^2$ for all $|s|\geq \xi_\tau $. Thus,
 \[
 o_n(1)+\frac{1}{2} \geq \int_{\{|w_n(x,0)|\geq \xi_\tau \}} \frac{K(x)F(w_n(x,0))}{|w_n(x,0)|^2} z_n^2(x,0)   \ud x 
 \geq 
 \tau \int_{\R^N}K(x)  z_n^2(x,0)  \chi_{\{|z_n(x,0)|\geq \frac{\xi_\tau}{\|w_n\|} \}} \ud x.
 \]
Thus, by Fatou lemma, since $z_n^2(x,0)  \chi_{\{|z_n(x,0)|\geq \frac{\xi_\tau}{\|w_n\|} \}} \to z(x,0)$ a.e.,
for any $\tau>0$, we conclude
$$
\frac{1}{2} \geq \tau \int_{\R^N} K(x) z^2(x,0) \ud x.
$$
 Since $K>0$, it follows $z(x,0)=0$, by the arbitrariness of $\tau>0$ and the claim follows.
Now, let $B>0$.
Of course $B\|w_n\|^{-1}\in [0,1]$ eventually for $n\geq n_B$, for some $n_B\in\N$. Thus,
\[
J(t_n w_n) \geq J(B z_n) = 
\frac{B^2}{2}-
\int_{\mathbb{R}^N} K(x)F(B z_n(x,0)) \ud x,
\]
since $t_n$ is a maximum point. By Proposition~\ref{converge}, it follows  
$$
\int_{\R^N} K(x)F(B z_n(x,0))  \ud x \to \int_{\R^N} K(x)F(B z(x,0))=0,
$$
and we have  $J(t_nw_n)+o_n(1) \geq B^2/2,$
 which yields $\sup\{J(t_nw_n):n\in\N\}\geq B^2/2$, a contradiction if
$$
B=2\sqrt{\sup\{J(t_nw_n):n\in\N\}}\in (0,\infty).
$$ 
This concludes the proof.
\vskip3pt
\noindent{\bf Case $\lambda=1.$} Denote $J_\lambda = J.$ The boundedness of the $\{w_n\}$ in $E$  follows easily from (AR), since
$$
 o_n(1)+ c \geq  J(w_n)-\frac{1}{\theta}J'(w_n)(w_n)\geq (\frac{1}{2}-\frac{1}{\theta})\|w_n\|^2.
$$
This concludes the proof.
 \end{proof}

\noindent
The following Sobolev inequality can be found in \cite{colorado},
\begin{equation}\label{trace}
\int_{ \R^{N+1}_{+}}y^{1-s}|\nabla w|^2 \ud x\ud y\geq \S(s,N)\Big(\int_{ \R^{N}}| w(x,0)|^{2^{*}_{s}} \ud x\Big)^{\frac{2}{2^{*}_{s}}},\quad \text{for all $w \in X^{s}(\R^{N+1}_{+})$},
 \end{equation}
where 
$$
\S(s,N)= \frac{\Gamma(\frac{s}{2})\Gamma(\frac{1}{2}(N-s)))(\Gamma(N))^{\frac{s}{N}}}{2\pi^{\frac{s}{N}}\Gamma(\frac{1}{2}(2-s))
\Gamma(\frac{1}{2}(N+s))(\Gamma(\frac{1}{2}N))^{\frac{s}{N}}}.
$$
This constant is achieved on the family of functions \cite{colorado, coti, servadei2} $w_{\epsilon}=E_{s}(u_{\epsilon})$ (by \cite{talenti}
for $s=2$), where
$$
u_{\epsilon}(x)=
\frac{\epsilon^{\frac{N-s}{2}}}{(|x|^2 + \epsilon^2)^{\frac{N-s}{2}}} , \,\quad
\epsilon > 0.
$$
Furthermore, take 
 $\phi(x,y)=\phi_{0}(|(x,y)|),$ where $\phi_0 \in C^{\infty}(0,\infty)$ is a non increasing 
cut-off  such that
$$
\phi_0(s)=1\quad \mbox{if}\ s\in [0, 1/2],\quad \phi_0(s)=0,\quad \mbox{if}\ s\geq 1.
$$
Let $\phi w_{\epsilon}$ which belongs to $X^{s}(\R^{N+1}_{+}).$
By \cite[Lemma 3.8]{barrios} (which is formulated on a bounded domain $\Omega$, but which holds with the very same proof
when taking $\Omega=\R^N$), we have 


\begin{lemma}[concentration]
\label{Lemma 3.8}  
The family $\{\phi w_{\epsilon}\}$, and its trace on $\{y=0\}$, namely, $\phi u_{\epsilon},$ satisfy
\begin{equation}\label{3.40}
     \|\phi w_{\epsilon}\|^{2}_{X^{s}}= \|w_{\epsilon}\|^{2}_{X^{s}}+ \O(\epsilon^{N-s}),
\end{equation}
\begin{equation}\label{3.41}
\|\phi u_{\epsilon}\|^{2}_{L^2}=\left\{ \begin{array}{lcr}
\O(\epsilon^{s}) &\mbox{if}& N >2s,\\  
 \O(\epsilon^{s}\log(1/\epsilon)) &\mbox{if}& N= 2s,\\ 
 \O(\epsilon^{N-s})&\mbox{if}& N <2s,\\  
\end{array}\right.
\end{equation}
for $\epsilon>0$ small enough.
Define $\eta_{\epsilon}=\phi w_{\epsilon} /\|\phi u_{\epsilon}\|_{L^{2^{*}_{s} }} $, then
\begin{equation} \label{3.53.1}
                 \|\eta_{\epsilon}\|^{2}_{X^{s}}=  k_{s}\S(s,N)  + \O(\epsilon^{N-s}),
\end{equation}
\begin{equation}\label{3.53.2}
\|\eta_{\epsilon}(x,0)\|^{2}_{L^2}=\left\{ \begin{array}{lcr}
\O(\epsilon^{s})&\mbox{if}& N >2s,\\   
 \O(\epsilon^{s}\log({1/\epsilon}) )&\mbox{if}& N = 2s,\\ 
\O(\epsilon^{N-s})&\mbox{if}& N <2s.\\  
\end{array}\right.
\end{equation}
and 
\begin{equation}
\label{stima-q}
\|\eta_\eps(x,0)\|_{L^q}^q= 
\left\{ \begin{array}{lll}
\O(\eps^{\frac{2N-(N-s)q}{2}})&  &\text{if $q>\frac{N}{N-s}$ (or $N\geq 2s$),}  \\   
 \O(\eps^{\frac{(N-s)q}{2}})  & &\text{if $q<\frac{N}{N-s}$}.\\  
\end{array}\right.
\end{equation}
Here with the notation $a_\eps=\O(b_\eps)$ we mean that $a_\eps/b_\eps$ is uniformly bounded with respect to $\eps$.
\end{lemma}

\begin{remark}\rm
We remark that, actually, except \eqref{3.40} and \eqref{3.53.1}, the other estimates follow exactly as in local case (see \cite{BN}), because
in these cases, we know the explicity expression for $u_{\epsilon}.$ While, for $w_{\epsilon},$ except for $s=1$ (see \cite{tanhalf})
and $s=2$ (local case), the explicit expressions are not available. But, in \cite{barrios}, the authors were clever to overcome
this difficulty, by exploring some  properties of the  Poisson kernel. The $s$-harmonic extension of the $u_{\epsilon}$ has the following explicit expression
$$
w_{\epsilon}(x,y)=P^{s}_{y}\ast u_{\epsilon}(x)=C_{N,s}y^{s}\int_{\R^N}\frac{u_{\epsilon}(\xi)}{(|x-\xi|^2 + y^2)^{\frac{N+s}{2} } } \ud\xi,
 \quad \text{for some $C_{N,s}>0$.}
 $$
Noticing that as $u_{\epsilon}$ and $P^{s}_{y}$ are self-similar functions, namely
$$
u_{\epsilon}(x)=\epsilon^{\frac{s-N}{2}}u_1(x/\epsilon),\qquad 
P^{s}_{y}(x)=\frac{1}{y^N}P^{s}_{1}(x/|y|)=\frac{y^s}{(|x|^2 + y^2)^{\frac{N+s}{2}}},
$$
then 
$$
w_{\epsilon}(x,y)=\epsilon^{\frac{s-N}{2}}w_1(x/\epsilon,y/\epsilon).
$$
Exploiting this fact, they estimate as follows
$$
\int_{\R^{N+1}_{+}} y^{1-s}  w_{\epsilon}\phi \nabla \phi\cdot \nabla w_{\epsilon} \ud x \ud y \leq \O(\epsilon^{ N-s}),\quad
\int_{\R^{N+1}_{+}} y^{1-s}|w_{\epsilon}\nabla \phi|^2 \ud x \ud y \leq \O(\epsilon^{ N-s}).
$$
Combining these inequalities, \eqref{3.40} holds.
The inequality \eqref{3.53.1} comes as a consequence.
Concerning \eqref{stima-q}, we justify it in the case $q<N/(N-s)$, the opposite case being similar. We have
\begin{align*}
\|\phi u_\eps\|_{L^q}^q& =\int_{{\mathbb R}^N} |\phi|^q|u_\eps|^q  \ud x\geq \int_{B(0,\frac{1}{2})} |u_\eps|^q \ud x\\
&=\eps^{\frac{(N-s)q}{2}}\int_{B(0,\frac{1}{2})} \frac{1}{(\eps^2+|x|^2)^{\frac{(N-s)q}{2}}} \ud x \\
&=\eps^{\frac{(N-s)q}{2}}\int_{B(0,\frac{1}{2\eps})} \frac{1}{\eps^{(N-s)q}(1+|y|^2)^{\frac{(N-s)q}{2}}} \eps^N \ud y \\
&=\eps^{N-\frac{(N-s)q}{2}}\int_{B(0,\frac{1}{2\eps})} \frac{1}{(1+|y|^2)^{\frac{(N-s)q}{2}}} \ud y \\
&=\eps^{N-\frac{(N-s)q}{2}}C\int_0^{1/(2\eps)}\frac{1}{(1+\varrho^2)^{\frac{(N-s)q}{2}}} \varrho^{N-1} \ud \varrho \\
&=\eps^{N-\frac{(N-s)q}{2}}C\left(\int_0^{1}\frac{1}{(1+\varrho^2)^{\frac{(N-s)q}{2}}} \varrho^{N-1} \ud \varrho+
\int_1^{1/(2\eps)}\frac{1}{(1+\varrho^2)^{\frac{(N-s)q}{2}}} \varrho^{N-1} \ud\varrho\right) \\
&\geq \eps^{N-\frac{(N-s)q}{2}}\Big(C+
C\int_1^{1/(2\eps)}\frac{1}{\varrho^{(N-s)q-N+1}} \ud\varrho\Big) \\
&\geq \eps^{N-\frac{(N-s)q}{2}}\Big(C+
C\eps^{(N-s)q-N})\Big) 
\geq C\eps^{\frac{(N-s)q}{2}},
\end{align*}
for $\eps>0$ small enough. Analogously, we get
\begin{align*}
\|\phi u_\eps\|_{L^q}^q& =\int_{{\mathbb R}^N} |\phi|^q|u_\eps|^q\ud x\leq \int_{B(0,1)} |u_\eps|^q \ud x\\
&=\eps^{\frac{(N-s)q}{2}}\int_{B(0,1)} \frac{1}{(\eps^2+|x|^2)^{\frac{(N-s)q}{2}}} \ud x \\
&=\eps^{\frac{(N-s)q}{2}}\int_{B(0,\frac{1}{\eps})} \frac{1}{\eps^{(N-s)q}(1+|y|^2)^{\frac{(N-s)q}{2}}} \eps^N \ud y \\
&=\eps^{N-\frac{(N-s)q}{2}}\int_{B(0,\frac{1}{\eps})} \frac{1}{(1+|y|^2)^{\frac{(N-s)q}{2}}} \ud y \\
&=\eps^{N-\frac{(N-s)q}{2}}C\int_0^{1/\eps}\frac{1}{(1+\varrho^2)^{\frac{(N-s)q}{2}}} \varrho^{N-1} \ud \varrho \\
&=\eps^{N-\frac{(N-s)q}{2}}C\left(\int_0^{1}\frac{1}{(1+\varrho^2)^{\frac{(N-s)q}{2}}} \varrho^{N-1} \ud\varrho+
\int_1^{1/\eps}\frac{1}{(1+\varrho^2)^{\frac{(N-s)q}{2}}} \varrho^{N-1} \ud\varrho\right) \\
&\leq \eps^{N-\frac{(N-s)q}{2}}\Big(C+
C\int_1^{1/\eps}\frac{1}{\varrho^{(N-s)q-N+1}} \ud\varrho\Big) \\
&\leq \eps^{N-\frac{(N-s)q}{2}}\Big(C+
C\eps^{(N-s)q-N})\Big) \leq C\eps^{\frac{(N-s)q}{2}},
\end{align*}
for $\eps>0$ small enough.
Since $\|\phi u_{\epsilon}\|_{L^{2^{*}_{s} }}$ converges to a positive constant,
the assertion follows.
\end{remark}

\noindent
The following result will be crucial for the proof of our main result

\begin{lemma}[MP energy bound]
\label{crucial}
Let $\lambda=1$ and let {\rm (f1)-(f2)-(f3)$'$} hold. Then $c< \frac{s}{2N}(k_{s}\S(s,N))^{N/s}$.
\end{lemma}
\begin{proof}
By definition of $c$, it is sufficient to prove that there exists $\epsilon >0$ small enough that
$$
\sup_{t\geq 0} J(t\eta_{\epsilon}) < \frac{s}{2N}(k_{s}\S(s,N))^{N/s}, \quad  J=J_1.
$$
By definition of $J,$ we have
$$
J(t\eta_{\epsilon})=\frac{t^2}{2}\|\eta_{\epsilon}\|^{2} - \int_{\mathbb{R}^N} K(x) F(t\eta_{\epsilon}(x,0))   \ud x  
- \frac{t^{2^{*}_{s}}}{2^{*}_{s}}.   
$$
By the assumptions of $f$, there exist $q\in (2,2^*_s)$ and  $C_0>0$ with
$F(s)\geq C_0 s^q$ for any $s\in \R^+$. Then
$$
J(t\eta_{\epsilon})\leq \psi(t),   \,\,\quad 
\psi(t)=\frac{t^2}{2}\|\eta_{\epsilon}\|^2- C_0t^q \int_{\mathbb{R}^N} |\eta_{\epsilon}(x,0)|^q  \ud x  
- \frac{t^{2^{*}_{s}} }{2^{*}_{s}}.   
$$
Since $\psi(t)\to-\infty$ as $t \to+\infty$, we have $\sup\{\psi(t):t\geq 0\}=\psi(t_{\epsilon})$
for some $ t_{\epsilon}>0$, so that
$$
\|\eta_{\epsilon}\|^2 - C_0q t^{q-2}_\eps\int_{\mathbb{R}^N} |\eta_{\epsilon}(x,0))|^q  \ud x  
=t^{2^{*}_{s}-2}_{\epsilon},
$$
which yields $\sigma_0\leq t_{\epsilon}\leq \|\eta_{\epsilon}\|^{\frac{2}{2^{*}_{s}-2}}\leq K_0$ for some  
$\sigma_0,K_0>0$ independent of $\eps$, in view of Lemma~\ref{Lemma 3.8} and the above equality. Since the map
$$
[0, \|\eta_{\epsilon}\|^{\frac{2}{2^{*}_{s}-2}}]\ni t\mapsto \frac{t^2}{2}\|\eta_{\epsilon}\|^2 - \frac{t^{2^{*}_{s}} }{2^{*}_{s}},
$$
increases, we get for some universal constant $C>0$,
\begin{align*}
 \sup_{\R^+}\psi &\leq 
 \frac{s}{2N}\Big(\|\eta_{\epsilon}\|^2_{X^s}+\int_{\R^N}V(x)\eta_\eps(x,0)^2 \ud x\Big)^{N/s}-C_0C \|\eta_{\epsilon}(x,0))\|^{q}_{L^q} \\
 &\leq \frac{s}{2N}\Big(k_{s}\S(s,N)  + \O(\epsilon^{N-s})+\int_{\R^N}V(x)\eta_\eps(x,0)^2 \ud x\Big)^{N/s}-C_0C \|\eta_{\epsilon}(x,0))\|^{q}_{L^q}.
 \end{align*}
Now, by the elementary inequality $(a+b)^\alpha\leq a^\alpha+\alpha (a+b)^{\alpha-1}b$, $\alpha\geq 1$ and $a,b>0$, we get by \eqref{3.53.1}
\begin{align*}
 \sup_{\R^+}\psi &\leq 
 \frac{s}{2N}(k_{s}\S(s,N))^{N/s}+ \O(\epsilon^{N-s})+C\int_{\R^N}V(x)\eta_\eps(x,0)^2 \ud x-C_0C \|\eta_{\epsilon}(x,0))\|^{q}_{L^q} \\
 & \leq   \frac{s}{2N}(k_{s}\S(s,N))^{N/s}+ \O(\epsilon^{N-s})+C\|\eta_\eps(x,0)\|_{L^2}^2-C_0C \|\eta_{\epsilon}(x,0))\|^{q}_{L^q}
 \end{align*}
\noindent
$\bullet$ In the case $N>2s$, by means of \eqref{3.53.2} and \eqref{stima-q}, we get
\begin{equation*}
 \sup_{\R^+}\psi \leq   \frac{s}{2N}(k_{s}\S(s,N))^{N/s}+ \O(\epsilon^{N-s})+\O(\epsilon^{s})-C_0\O(\eps^{\frac{2N-(N-s)q}{2}}).
 \end{equation*}
 Since ${\frac{2N-(N-s)q}{2}}<s<N-s$, we get the conclusion for $\eps$ sufficiently small. 
\vskip2pt
\noindent 
 $\bullet$ If $N=2s$ and $2<q<2^*_s=4$, by \eqref{3.53.2} and \eqref{stima-q}, we get
 \begin{equation*}
  \sup_{\R^+}\psi \leq   \frac{s}{2N}(k_{s}\S(s,N))^{N/s}+ \O(\epsilon^{s}(1+\log(\epsilon^{-1})) )-C_0\O(\eps^{\frac{2N-sq}{2}}).
  \end{equation*}
 Since it holds
 $$
 \lim_{\eps\to 0}\frac{\eps^{\frac{2N-sq}{2}}}{\epsilon^{s}\log(\epsilon^{-1}) }=+\infty,
 $$
 again can we get the conclusion, for $\eps$ sufficiently small.
 \vskip2pt
 \noindent 
  $\bullet$ If $s<N<2s$ and $\frac{N}{N-s}<q<2^*_s$, by \eqref{3.53.2} and \eqref{stima-q}, we get
  \begin{equation*}
   \sup_{\R^+}\psi \leq   \frac{s}{2N}(k_{s}\S(s,N))^{N/s}+ \O(\epsilon^{N-s})-C_0\O(\eps^{\frac{2N-(N-s)q}{2}}).
   \end{equation*}
   Since ${\frac{2N-(N-s)q}{2}}<N-s$ means $q>\frac{2s}{N-s}(>2)$, we get the conclusion for $\eps$ sufficiently small. 
 \vskip2pt
 \noindent 
  $\bullet$ If $s<N<2s$ and $2<q<\frac{N}{N-s}$, by \eqref{3.53.2} and \eqref{stima-q}, we get
  \begin{equation*}
   \sup_{\R^+}\psi \leq   \frac{s}{2N}(k_{s}\S(s,N))^{N/s}+ \O(\epsilon^{N-s})-C_0\O(\eps^{\frac{(N-s)q}{2}}),
   \end{equation*}
and for $C_0=\eps^{-\vartheta}$ with $\vartheta>\frac{(N-s)(q-2)}{2}$, we get the conclusion. This concludes the proof. 
\end{proof}

\section{Proof of Theorem~\ref{mainsub} completed}
\label{prova1}
In light of Lemma~\ref{MPG}, there exists a {\em Cerami} sequence $\{w_n\}\subset E$ for $J=J_0$. 
From Lemma~\ref{cerami} it follows that $w^-_n\to 0$ in $E$ as $n\to\infty$ and that $\{w_n\}$ is bounded and has 
a nonnegative weak limit $w\in E$. 
By Proposition \ref{converge},  it follows that $w$ is a weak nonnegative solution, to which a weak solution $u \in H^{s/2}(\R^N)$ to \eqref{PS} corresponds. 
We have $u>0$ if $u\neq 0$. In fact, if $u(x_0)=0$ for some $x_0\in\R^N$, then $(-\Delta)^{s/2} u(x_0)=0$
and by the representation formula \cite{nezza}
$$
(-\Delta)^{s/2} u(x)=-\frac{c(N,s/2)}{2}\int_{\R^N}\frac{u(x+y)+u(x-y)-2u(x)}{|x-y|^{N+s}} \ud y,
$$
one obtains that, at $x_0$, that
$$
\int_{\R^N} \frac{u(x_0+y)+u(x_0-y)}{|x_{0}-y|^{N+s}}\ud y=0,
$$
yielding $u=0$, a contradiction. Let us prove that, indeed, $u\neq 0$. We prove that $w=E_s(u)\not\equiv 0$.
In fact, $(w_n)$ converges  to $w$ strongly in $E.$
Indeed, up to a subsequence, $w_n \rightharpoonup w$ in $E$ as $n\to\infty$, and since $J'(w_n)(w_n)=o_n(1),$ we have, again 
by virtue of Proposition \ref{converge},
\begin{align*}
\lim_{n\to \infty} \|w_n\|^2 & =\lim_{n\to\infty}\int_{\R^N}K(x)f(w_n(x,0))w_n(x,0)dx \\
&=\int_{\R^N}K(x)f(w(x,0))w(x,0)dx=\|w\|^2,
\end{align*}
that is, $ w_n \rightarrow w$ in $E.$  Hence $J(w)=c$ and $J'(w)=0,$ this implies that $w\not\equiv 0.$ \qed

\section{Proof of Theorem~\ref{main} completed}
\label{prova2}
\noindent 
In light of Lemma~\ref{MPG}, there exists a {\em Cerami} sequence $\{w_n\}\subset E$ for $J=J_1$. 
From Lemma~\ref{cerami} it follows that $w^-_n\to 0$ in $E$ as $n\to\infty$ and that $\{w_n\}$ is bounded and has 
a nonnegative weak limit $w\in E$. 
By Proposition \ref{converge},  it follows that $w$ is a weak nonnegative solution, to which a weak solution $u \in H^{s/2}(\R^N)$ to \eqref{PS} corresponds. 
We have $u>0$ if $u\neq 0$, arguing as in Section~\ref{prova1}. Let us prove that, indeed, $u\neq 0$. We prove that $w=E_s(u)\not\equiv 0$.
By virtue of \eqref{ceramiS}, we have
\begin{align*}
& \frac{1}{2}\|w_n\|^2 - \int_{\mathbb{R}^N} K(x) F(w(x,0))   \ud x  
- \frac{1}{2^{*}_{s}}\int_{\mathbb{R}^N} w_n^+(x,0)^{2^{*}_{s}}   \ud x=c+o_n(1), \\
& \|w_n\|^2 - \int_{\mathbb{R}^N} K(x) f(w(x,0))w(x,0)  \ud x - \int_{\mathbb{R}^N} w_n^+(x,0)^{2^{*}_{s}} \ud x=o_n(1).
\end{align*}
Suppose, by contradiction, that $w\equiv 0$. Then, we entail
$$
\Big(\frac{1}{2}-\frac{1}{2^{*}_{s}}\Big)\|w_n\|^2=c+o_n(1), 
$$
which combined with $\|w_n\|^2=\|w_n(x,0)\|_{2^*_s}^{2^*_s}+o_n(1)$ as $n\to\infty$ and the Sobolev inequality
$$
\|w_n\|^2\geq \int_{ \R^{N+1}_{+}}k_{s}y^{1-s}|\nabla w|^2 \ud x\ud y\geq k_s \S(s,N) \|w_n(x,0)\|_{2^*_s}^2
$$
implying
$$
c=\lim_n J(w_n)=\Big(\frac{1}{2}-\frac{1}{2^{*}_{s}}\Big)\lim_n \|w_n\|^2\geq \frac{s}{2N}(k_{s}\S(s,N))^{N/s}.
$$
This contradicts Lemma~\ref{crucial}. Hence $w\not\equiv 0$ and the proof is complete. \qed

\bigskip
\bigskip

\end{document}